\documentclass[12pt]{article}

\usepackage{latexsym}
\usepackage{graphicx}
\usepackage{enumitem}
\usepackage{amsmath}
\usepackage{amssymb}
\usepackage{amsthm}
\usepackage{multirow}
\usepackage{color}
\usepackage{url}
\newtheorem{theorem}{Theorem}[section]
\newtheorem{definition}[theorem]{Definition}

\newtheorem{lemma}[theorem]{Lemma}

\newtheorem{corollary}[theorem]{Corollary}

\newtheorem{problem}{Problem}

\title{On lattice tilings of $\mathbb{Z}^n$ by limited magnitude error balls $\mathcal{B}(n,2,k_{1},k_{2})$ with $k_1>k_2$}

\author{Ka Hin Leung$^{\text{a,}}$\thanks{Email address:  matlkh@nus.edu.sg.}, Ran Tao$^{\text{b,}}$\thanks{Email address:  rtao@sdu.edu.cn.}, Daohua Wang$^{\text{c,}}$\thanks{Email address:  wangdaohua@xidian.edu.cn.}~ and  Tao Zhang$^{\text{d,}}$\thanks{Email address:  zhant220@163.com.}\\
	\footnotesize $^{\text{a}}$ Department of Mathematics, National University of Singapore, Kent Ridge, Singapore 119260, Republic of Singapore. \\
	\footnotesize $^{\text{b}}$ School of Cyber Science and Technology, Shandong University, Qingdao 266237, China.\\
	\footnotesize $^{\text{c}}$ School of Cyber Engineering, Xidian University, Xi’an 710126, China.\\
	\footnotesize $^{\text{d}}$ Institute of Mathematics and Interdisciplinary Sciences, Xidian University, Xi'an 710126, China.\\}
\begin{document}
	\date{}

\maketitle

\begin{abstract}
Lattice tilings of $\mathbb{Z}^n$ by limited-magnitude error balls correspond to linear perfect codes under such error models and play a crucial role in flash memory applications. In this work, we establish three main results. First, we fully determine the existence of lattice tilings by $\mathcal{B}(n,2,3,0)$ in all dimensions $n$. Second, we completely resolve the case $k_1=k_2+1$. Finally, we prove that for any integers $k_1>k_2\ge0$ where $k_1+k_2+1$ is composite, no lattice tiling of $\mathbb{Z}^n$ by the error ball $\mathcal{B}(n,2,k_1,k_2)$ exists for sufficiently large $n$.
	
	\medskip
	
	\noindent {{\it Keywords\/}: Lattice tiling, perfect code, limited magnitude error ball.}
	
	\smallskip
	
	\noindent {{\it AMS subject classifications\/}: 52C22, 11H31, 11H71.}
\end{abstract}

\section{Introduction}
For integers $n\ge t\ge1$ and $k_{1}\ge k_{2}\ge0$,  the limited-magnitude $(n,t,k_{1},k_{2})$-error ball is defined as
\[\mathcal{B}(n,t,k_{1},k_{2}):=\{{\bf{a}}=(a_1,a_2,\dots,a_n)\in\mathbb{Z}^n:\ a_i\in[-k_{2},k_{1}],\ \text{wt}({\bf{a}})\le t\},\]
where $\text{wt}({\bf{a}})$ denotes the Hamming weight of ${\bf{a}}$. A subset $C\subset\mathbb{Z}^n$ is said to form a tiling of $\mathbb{Z}^n$ by $\mathcal{B}(n,t,k_{1},k_{2})$ if the translates
\[\mathcal{T}=\{\mathcal{B}(n,t,k_{1},k_{2})+c: c\in C\}\]
 partition $\mathbb{Z}^n$. When $C$ constitutes a lattice, the tiling is referred to as a lattice tiling.

In flash memory systems, data is stored using integer vector representations, making the limited magnitude error model particularly relevant \cite{CSBB10}. Within asymmetric error correction frameworks, symbol-level errors $|a-b|$ are constrained by predefined thresholds, while codeword-level errors affect only a bounded number of coordinates. Under this model, error-correcting codes correspond to packings of $\mathbb{Z}^n$ by $\mathcal{B}(n,t,k_{1},k_{2})$, whereas perfect codes correspond to tilings. Crucially, linear perfect codes are equivalent to the existence of lattice tilings by such error balls.

Perfect codes, renowned for their elegant structure and optimality in code size \cite{E22}, motivate the following fundamental question:

\begin{problem}
	For given $n\ge t\ge1$ and $k_{1}\ge k_{2}\ge0$, does there exist a lattice tiling of $\mathbb{Z}^{n}$ by $\mathcal{B}(n,t,k_{1},k_{2})$? Equivalently, does a linear perfect code exist for limited magnitude error ball $\mathcal{B}(n,t,k_{1},k_{2})$?
\end{problem}

When $t=n$, the error ball $\mathcal{B}(n,t,k_{1},k_{2})$ reduces to the hypercube $[-k_{2},k_{1}]^n$, prompting our focus on the non-trivial regime $1\le t\le n-1$.

Tilings and packings of $\mathbb{Z}^{n}$ by single-coordinate error balls ($t=1$) have been extensively studied due to their mathematical depth and practical applications. Notable research has been conducted on the cross shape $\mathcal{B}(n,1,k,k)$ and semi-cross shape $\mathcal{B}(n,1,k,0)$, as explored in
 \cite{HS86,KBE11,KLNY11,KLY12,S67,S84,SS94,S86,S87,T98,W95}. Building upon these studies, Schwartz \cite{S12} introduced the quasi-cross shape $\mathcal{B}(n,1,k_{1},k_{2})$, which has since aroused widespread interest \cite{S14,XL20,XL21,YKB13,YZZG20,ZG16,ZG18,ZZG17}. 
 
 For multi-coordinate errors ($t\ge2$), research remains limited. The tiling and packing properties of $\mathbb{Z}^{n}$ with  $\mathcal{B}(n,n-1,k,0)$ have been explored in \cite{BE13,KLNY11,S90}. More generally, the problem of tiling and packing for arbitrary parameters was investigated in \cite{WS22,WWS21}. Notably, Wei and Schwartz \cite{WS22} provided a classification of lattice tilings of $\mathbb{Z}^{n}$ using $\mathcal{B}(n,2,1,0)$ and $\mathcal{B}(n,2,2,0)$. Additionally, Zhang et al. \cite{ZLG2023} established the existence of lattice tilings of $\mathbb{Z}^{n}$ by $\mathcal{B}(n,2,1,1)$. Recently, Guan, Wei and Xiang \cite{GW25} demonstrated that no lattice tiling of $\mathbb{Z}^n$ by $\mathcal{B}(n,2,k_1,k_1-1)$ exists for sufficiently large $n$.
 
  This work advances the study of lattice tilings for two-coordinate errors ($t=2$). To illustrate the basic idea of our proofs,  we first provide a complete classification of lattice tilings by $\mathcal{B}(n,2,3,0)$.
\begin{theorem}\label{mainthm3}
	For $n\ge3$,  a lattice tiling of $\mathbb{Z}^n$ by $\mathcal{B}(n,2,3,0)$ exists if and only if $n=3$.
\end{theorem}

\medskip

Next, we will show how to provide a complete classification of lattice tilings by an infinite family $\mathcal{B}(n,2, k ,k-1)$.

\begin{theorem}\label{mainthm5}
	For any $n\ge3$ and $k\ge2$, there does not exist a lattice tiling of $\mathbb{Z}^n$ by $\mathcal{B}(n,2,k,k-1)$.
\end{theorem}

Unfortunately, we cannot solve all general cases completely. But we show a weaker version as follows: 

\begin{theorem}\label{mainthm1}
	Let $k_1>k_2\ge0$ with $k_1+k_2\ge3$, and $M$ be a positive integer. Suppose that all integers in the interval $[k_1+k_2+1,k_1+k_2+M]$ are composite. Let $p$ denote the smallest prime divisor of $k_1+k_2+1$, and define
	\begin{align*}
		A=&k_1+(4M-1)k_2,\\
		B=&(4M+8k_1+3\sqrt{p})(k_1-k_2)+2k_2(k_2+1)+2(k_1-k_2)^2(k_1-k_2+1)+(3+2k_2)M(M+1).
	\end{align*} 
 If $n\ge\lfloor\frac{B}{A}\rfloor+1$, then no lattice tiling of $\mathbb{Z}^n$ by $\mathcal{B}(n,2,k_1,k_2)$ exists.
\end{theorem}

As an immediate consequence, we obtain the following corollary:
\begin{corollary}\label{coro1}
	Let $k_1>k_2\ge0$ with $k_1+k_2\ge3$, and suppose that $k_1+k_2+1$ is composite. Let $p$ be the smallest prime divisor of $k_1+k_2+1$. If
	\[n\ge 2(k_1-k_2)^2+12(k_1-k_2)+ 2k_2+8.,\]
	 then a lattice tiling of $\mathbb{Z}^n$ by $\mathcal{B}(n,2,k_1,k_2)$ does not exist.
\end{corollary}

Note that in Theorem \ref{mainthm1}, we impose a condition $k_1+k_2+1$ not a prime. 
This is essential as if $k_1+k_2+1$ is indeed a prime $p$, then our method does not work in case $|G|$ is a $p$-power. However, such cases do not seem to happen for large $n$.
In case $k_1+k_2+1=5$, it seems $|G|$ is $5$-power only when $n=2$. 

The remainder of this paper is organized as follows. Section~\ref{pre} introduces the necessary group ring foundations and derives representations for lattice tilings. Sections~\ref{sub3}--\ref{gen} present the proofs of our main theorems. Finally, Section~\ref{conclu} concludes the paper.

\section{Preliminaries}\label{pre}
Let $\mathbb{Z}[G]$ denote the group ring of a finite abelian group $G$ (written multiplicatively) over the integers $\mathbb{Z}$. Each element $A\in\mathbb{Z}[G]$ admits a formal sum representation:
\[ A=\sum_{g\in G}a_gg,\] 
where $a_g\in\mathbb{Z}$. For any $A=\sum_{g\in G}a_gg$ and integer $t$, we define
\[A^{(t)}=\sum_{g\in G}a_gg^t.\]
Standard group ring operations are defined componentwise:
\[\sum_{g\in G}a_gg\pm\sum_{g\in G}b_gg=\sum_{g\in G}(a_g\pm b_g)g,\]
\[\sum_{g\in G}a_gg\sum_{g\in G}b_gg=\sum_{g\in G}(\sum_{h\in G}a_{h}b_{h^{-1}g})g,\]
and for scalar multiplication with $\lambda\in\mathbb{Z}$:
\[\lambda\sum_{g\in G}a_gg=\sum_{g\in G}(\lambda a_g)g,\]
where $\lambda\in\mathbb{Z}$ and $\sum_{g\in G}a_gg,\sum_{g\in G}b_gg\in\mathbb{Z}[G]$.

The following fundamental theorem bridges lattice tilings with finite abelian group structures:
\begin{theorem}{\rm{\cite{HA12}}}\label{tilinggroup}
	A subset $V\subset\mathbb{Z}^{n}$ admits a lattice tiling $\mathcal{T}$ of $\mathbb{Z}^n$ if and only if there exists:
	\begin{enumerate}
		\item [(1)] A finite abelian group $G$ with $|G|=|V|$;
		\item [(2)] A group homomorphism $\phi:\mathbb{Z}^n\rightarrow G$ such that $\phi|_{V}:V\rightarrow G$ is bijective.
	\end{enumerate}
\end{theorem}

To characterize lattice tilings by $\mathcal{B}(n,2,k_1,k_2)$ via an abelian group $G$, we introduce the following notations.

\begin{definition} 
\begin{itemize}
\item [(a)] For any integers $i, j$, we define 
\[ S(i,j)=\{g^ih^j: g, h\in T \mbox{ and } g\ne h \};\] 
\item [(b)] For any nonzero integers $a<b$, we define 
\[ [a,b]^*=\{ i\neq 0: a\leq i \leq b\}.\]
\end{itemize} 
\end{definition}

\begin{theorem}\label{main}
For integers $k_1>k_2\ge0$, 	there exists a lattice tiling of $\mathbb{Z}^n$ by $\mathcal{B}(n,2,k_1,k_2)$ if and only if there exists a finite abelian group $G$ of order $1+n(k_1+k_2)+\binom{n}{2}(k_1+k_2)^2$ and $T\subset G$ with $|T|=n$ 
such that one of the following equivalent conditions is satisfied:
	\begin{enumerate}
\item [(1)] 
		$ G=e+\sum_{i=1}^{n}\sum_{j\in I}t_{i}^j+\sum_{1\le i<j\le n}(\sum_{k\in I}t_i^k)(\sum_{l\in I}t_j^l)$,
		where $I=[-k_2,k_1]^{*}$.
\item [(2)] $G$ is a disjoint union of 
\[ G= \{e\}\cup  \bigcup_{i\in [-k_2, k_1]^*} T^{(i )} \cup \bigcup_{i,j\in [-k_2,k_1]^{*}, i\le j}S(i,j),\]
i.e. all the sets in the above equations are disjoint. 
	\end{enumerate}
\end{theorem}
\begin{proof}
	By Theorem~\ref{tilinggroup}, there exists a lattice tiling of $\mathbb{Z}^n$ by $\mathcal{B}(n,2,k_1,k_2)$ if and only if there are both an abelian group $G$ of order $1+n(k_1+k_2)+\binom{n}{2}(k_1+k_2)^2$ and a homomorphism $\phi:\mathbb{Z}^n\rightarrow G$ such that the restriction of $\phi$ to $\mathcal{B}(n,2,k_1,k_2)$ is a bijection.
	Let $e_i$, $i=1,2,\dots,n$, be a fixed orthonormal basis of $\mathbb{Z}^{n}$.
	Since the homomorphism $\phi$ is determined by the values $\phi(e_i)$, $i=1,\dots,n$, then above conditions are equivalent to there exists an $n$-subset $\{t_1,t_2,\dots,t_n\}\subset G$ (let $\phi(e_i)=t_i$) such that
	\begin{align*}
		G=&\{e\}\cup\{t_{i}^{j}:\ 1\le i\le n, j\in I\}\cup\{t_i^kt_j^l:\ 1\le i<j\le n, k,l\in I\}.
	\end{align*}
	Expressed in group ring notation, this becomes:
\[G=e+\sum_{i=1}^{n}\sum_{j\in I}t_{i}^j+\sum_{1\le i<j\le n}(\sum_{k\in I}t_i^k)(\sum_{l\in I}t_j^l).\]
Clearly, (1) and (2) are equivalent. 
\end{proof}

\medskip

From now on, we assume a lattice tiling of $\mathbb{Z}^n$ by $\mathcal{B}(n,2,k_1,k_2)$,  i.e. the existence of $G$ and $T$ such that conditions in Theorem~\ref{main} hold. We further assume that $k_1+k_2+1$ is a composite number. Next, we introduce the following combinatorial counters.

\begin{definition}
For any integers $i, j, m$, we define 
\begin{align*}
	&\psi(m,i,j):=|\{t\in T: t^m\in S(i,j)\}|,\\
	&\psi(m,i):=|\{t\in T: t^m\in T^{(i)}\}|,\\
	&\psi(m,0):=|\{t\in T: t^m=e\}|.
\end{align*}
\end{definition}

\begin{lemma}\label{lem1}
For any $i,j\in [-k_2,k_1]^{*}$, 
	\begin{itemize}
\item[(a)] $\psi(i,j)=0$ if $i\neq j$,
\item [(b)] $\psi(m,i,j)=0$ if $m\in [-k_2,k_1]^{*}$,
\item [(c)] 
\[ \psi(m,0)+\sum_{i\in [-k_2,k_1]^{*}}\psi(m,i)+\sum_{i,j\in [-k_2,k_1]^{*},\ i\le j}\psi(m,i,j)=n.\]

		\end{itemize}
\end{lemma}
\begin{proof} (a), (b), (c) follow easily from Theorem~\ref{main} (2).  \end{proof}

\medskip

Our next results can be applied to compute number of elements in the intersections between $S(i,j)$ and $T^{(\ell)}$. 

\begin{lemma}\label{lem2}
\begin{itemize}
\item [(a)] $|T^{(i)}|=n$ for any $i\in[-k_1-k_2, k_1+k_2]^{*}.$
\item [(b)] If $-k_2\leq i\leq k_1+k_2$ and $j\in [-k_2, k_1]^*$, then for any $g, h\in T$, 
$g^i=h^j$ implies $g=h$. 
\item [(c)] For any $i\in[-k_1-k_2, 0]$,  the sets 
\[ T^{(i)}, \ldots, T^{(i+k_1+k_2)}\]
are disjoint. In particular, $\psi(\ell, i)=0$ if $0<\ell-i\leq k_1+k_2$. 
\item [(d)] For any $\ell \in [k_1+1, 2k_1]$ and $-k_1-k_2\leq i \leq 0$, 
\[ \sum_{j=i}^{i+k_1+k_2} \psi(\ell, j)\leq n.\]
 \end{itemize}
\end{lemma}
\begin{proof}
By Theorem~\ref{main} (2), $|T^{(i)}|=|T^{(-i)}|=n$ for all $i\in[1, k_1].$ This proves (a). 

Suppose $g^i=h^j$ for distinct $g, h\in T$ with $j\in [-k_2, k_1]$. 
Clearly, $g=h$ if $i\in [-k_2, k_1]$ by (a) as $i$ must be $j$. 
In case $k_1+1\leq i \leq k_1+k_2$, we then have $g^{k_1}=h^j g^{-(i-k_1)}\in S(-(j-k_1), j)$. Note that $-(i-k_1)\in [-k_2, k_1]^*$. This is impossible and thus $g=h$. This proves (b). 

For (c), note that  $T^{(i)}\cap T^{(j)}=\emptyset $ if and only if $T^{(-i)}\cap T^{(-j)}=\emptyset $. Therefore, we need only to show (b) holds in case $0\geq i\geq -k_2$. 
It suffices to show that 
$T^{(\ell)}\cap T^{(j)}=\emptyset $ if $j>\ell \geq -k_2$ and  $j-\ell\leq k_1+k_2$. Assume
otherwise,  
$g^\ell=h^j$ for distinct $g, h\in T$. 
In view of Theorem~\ref{main} (2), we may also assume $j\geq k_1+1$. 
But then, we have 
$g^{\ell} h^{(-k_2)}=h^{(j-k_2)}$. Since $g\neq h$ and $k_1\geq j-k_2$, this is possible only if $\ell\geq k_1+1$. But then we have $g^{\ell-k_2} h^{(-k_2)}=h^{(j-k_2)}g^{(k_2)}$. Hence, $g^{\ell-k_2} h^{(-k_2)}\in S(-k_2, \ell-k_2)\cap S(-k_2, j-k_2)$. This is impossible as $j-k_2\neq \ell-k_2$.  

(d) follows easily from (c). 

 \end{proof}

\medskip

Using Lemma \ref{lem2}, we may then extend Lemma \ref{lem1} (c).

\begin{lemma}\label{lem6}
 If $m\in[k_1+1,k_1+k_2], i\in [-k_2,k_2]^{*}, j\in [-k_2,k_1]^{*}$ and both 
$m-i,m-j\le k_1+k_2$, then $\psi(m,i,j)=0$.
\end{lemma}
\begin{proof}
Suppose there exist $g,g_1,g_2\in T$ with $g_1\ne g_2$ such that
$g^m=g_1^ig_2^j.$ 

Suppose $g=g_1$. Then 
 $g\neq g_2$ and $g^{m-k_2-i}=g^{-k_2}g_2^{j}$. This is impossible as $-k_2\leq m-k_2-i\leq k_1$.  Suppose $g=g_2$. 
As argued before, we are done if $j\in [-k_2, k_2]^*$. So we may assume $j\geq k_2+1$. 
Then, $g^{m-j}=g_1^{i}$. Since $-k_2\leq m-j \leq k_1$, it follows from Lemma \ref{lem2} (a) that $m-j=i$ and $g=g_1$.  This is impossible.

Finally, we consider the case $g\notin \{g_1, g_2\}$. 
We then have $g^{m-k_2}g_1^{-i}=g^{-k_2}g_2^{j}\in S(-i, m-k_2)\cap S(-k_2, j)$. 
Note that $-i, m-k_2, -k_2, j\in [-k_2, k_1]^*$. Therefore, 
$S(-i, m-k_2)=S(-k_2, j)$ and $g_1=g$ or $g_2$, which is impossible. 
\end{proof}

\medskip

In later section,  we need to determine the size of $S(i,j)$ even if $j\notin [-k_2, k_1]^*$. 
Before that, we need some results concerning the group $G$. 

Let $p$ be a prime. Denote by $S_p(G)$ the Sylow $p$-subgroup of $G$. We then establish the following lemma.
\begin{lemma}\label{lem13}
	Let $p\in[2,k_1]$ be a prime. 
	 The rank of $S_p(G)$ satisfies the following bound:
	\[\text{rank}(S_p(G))\le\begin{cases}3,&\textup{ if } k_1+k_2=3;\\
		2,&\textup{ if }k_1+k_2\ge5.\end{cases}\]
\end{lemma}
\begin{proof}
	Suppose that $\text{rank}(S_p(G))=r\ge1$. Then, the order of the $p$-core of $G$ is given by
	\[|G^{(p)}|=\frac{1}{p^r}|G|.\]
	Observe that
	\[\{e\}\cup \bigcup_{i\in[-k_2,k_1]^{*}, p\mid i}T^{(i)}\cup \bigcup_{i,j\in[-k_2,k_1]^{*}, i\le j, p\mid i, p\mid j}S(i,j)\subseteq G^{(p)}. \]
By computing the number of elements in the set above, we derive the following inequality:
	\begin{align*}
		&\frac{1}{p^r}|G|=\frac{1}{p^r}(\binom{n}{2}(k_1+k_2)^2+n(k_1+k_2)+1)\\
				\ge&1+(\lfloor\frac{k_1}{p}\rfloor+\lfloor\frac{k_2}{p}\rfloor)(n+\binom{n}{2})+(n^2-n)\binom{\lfloor\frac{k_1}{p}\rfloor+\lfloor\frac{k_2}{p}\rfloor}{2}.
	\end{align*}
	From this, we obtain the bound
	\[p^r\le\frac{\binom{n}{2}(k_1+k_2)^2+n(k_1+k_2)+1}{\binom{n}{2}(\lfloor\frac{k_1}{p}\rfloor+\lfloor\frac{k_2}{p}\rfloor)^2+n(\lfloor\frac{k_1}{p}\rfloor+\lfloor\frac{k_2}{p}\rfloor)+1}.\]
Write $k_1+k_2= p y +x$ where $y= (\lfloor\frac{k_1}{p}\rfloor+\lfloor\frac{k_2}{p}\rfloor)$. Clearly, $0\leq x \leq 2p-2$ and $y\geq 1$ as $p\leq k_1$. 
In case $p\geq 7$ or $y\geq 2$,  then it is easy to check that 
$(k_1+k_2)^2+(k_1+k_2)< p^3 y^2$. 

Next, we consider the case $y=1$ and $p\leq5$.
If $p=2$ and $y=1$, $k_1=2$ and $k_1=1$, then $3^2+3< 2^4$. Therefore, $r\leq 3$.
For $p=3$ and $y=1$, $3\leq k_1\leq 5$ and $0\leq k_2\leq 2$. Note that  $k_1+k_2\not\equiv 1 \bmod 3$  as otherwise, $3\nmid |G|$. Therefore $k_1+k_2=4$. Clearly, $(k_1+k_2)^2+(k_1+k_2) < 3^3$. Thus, $r\leq 2$. 
In case $p=5$, 
$5\leq k_1\leq 9$ and $0\leq k_2\leq 4$. However, $5$ divides $|G|$ only when $k_1+k_2\equiv 4 \bmod 5$. Therefore, $k_1+k_2=9$ and again, we see that
$9^2+9<5^3$. Thus $r\leq 2$. \end{proof}

\medskip

\begin{lemma}\label{lem8}
Let $m$ be a composite integer in $[k_1+k_2+1,2k_1]$ and $m=p\ell $ where $p$ is the smallest prime divisor of $m$. 
\begin{itemize}
\item [(a)] Suppose $C=\max_{g\in T}|\{h\in T: g^{k_1+k_2+1}=h^{k_1+k_2+1}\}$, then
		\[C \le\begin{cases}3\sqrt{p},&\textup{ if } k_1+k_2\ge5 \text{ and }p\ge3;\\
			3,&\textup{ if } k_1+k_2\ge5 \text{ and }p=2;\\
			4,&\textup{ if } k_1+k_2=3.\end{cases}\]
\item [(b)]  For any composite integer $m\in[k_1+k_2+1,2k_1]$,   
\[\psi(m,0)\le\begin{cases}2,&\textup{ if } m=k_1+k_2+1\ge6;\\
			3,&\textup{ if }k_1+k_2=3.\end{cases}\]
\end{itemize}
\end{lemma}

\begin{proof}
Define $H$ as the maximal elementary $p$-subgroup of $G$.  
Write $T^{(\ell)}=\sum_{i=1}^{r}g_iA_i$ where each $Hg_i$'s are distinct $H$-cosets and we may assume $g_1=e$.   
Note that as $|T^{(\ell)}|=n$, then $\psi(m,0)=|A_1|$. Let 
\[ x=-\lfloor\frac{k_2}{\ell }\rfloor \mbox{ and } y=\lfloor\frac{k_1}{\ell}\rfloor. \]
We then consider 
\[ \{e\} \cup \bigcup_{j \in [x, y]^*}	g_i^j A_i^{( j)} \cup \bigcup_{ \alpha\neq \beta \in [x, y]^*} g_i^{(\alpha +\beta)} [A_i^{( \alpha)} A_i^{(\beta)} -A_i^{(\alpha+\beta)}] .\]

For (a), we assume $m=k_1+k_2+1$, then $C=max \{ |A_i|: i=1, \ldots r\}$. As $|m-(|x|+y)\ell|<2(\ell-1)$, we conclude that $|x|+y=p-1$. Note that $x$ might be $0$. In any case, there are $p-1$ nonzero integers in $[x, y]^*$. 
Set $\alpha+\beta=y+x$, we see that 
\[ \bigcup_{ \alpha \in [x, y]^*} g_i^{y+x} [A_i^{(\alpha)} A_i^{(p-1-\alpha)} -A_i^{(y+x)}]\subset g_i^{(y+x)} H. \]
Therefore, we obtain the inequality  $\binom{p-1}{2}|A_i|(|A_i|-1)\le p^2$ if $p\neq 2$. 
If $p\geq 3$ and $k_1+k_2+1\geq 5$ we then obtain 
\[|A_i| \le \frac{(p-2)+\sqrt{(p-2)(8p^2+p-2)}}{2(p-2)}\le3\sqrt{p}.\]
In case $p=2$ and $k_1+k_2\geq 5$, we then obtain $|A_i|(|A_i|-1)\leq 4$. Thus $|A_i|\leq 3$.
Whereas if $k_1+k_2=3$, we obtain $|A_i|(|A_i|-1)\leq 8$ instead. In that case, $|A_i|\leq 4$. Recall that  $C=\max \{ |A_i|: i=1, \ldots r\}$. This proves (a). 

Next, we find a bound for $|A_1|$. We consider 
\[ \{e\} \cup \bigcup_{j \in [x, y]^*}	g_1^{j} A_1^{(j)} \cup \bigcup_{ \alpha\neq \beta \in [x, y]^*} g_1^{\alpha +\beta} [A_1^{(\alpha)} A_1^{(\beta)} -A_1^{(\alpha+\beta)}] \subset H .\]
As the set on the left contains 
\[ 1+(|x|+y)|A_1|+\frac{(|x|+y)^2}{2} \binom{|A_1|}{2}\]
elements.
In case $p$ is odd or $k_1+k_2+1\geq 5$, 
\[ 1+(|x|+y)|A_1|+\frac{(|x|+y)^2}{2} \binom{|A_1|}{2}\leq p^2.\]
If $m=k_1+k_2+1$, then $|x|+y=p-1$. Thus, $|A_1|\leq 2$ and $\psi(k_1+k_2+1, 0)\leq 2$. In case $m> k_1+k_2+1$, then it is clear that $|x|+y\geq (p-1)/2$ and thus $|A_1|\leq 3$. 

Finally, if $p$ is even and $k_1+k_2=3$, then we consider only the set
$\{e\} \cup A_1\cup  (A_1^2-A_1^{(2)})$. 
We then obtain $1+|A_1|+\binom{|A_1|}{2}\le8$ and thus $\psi(m,0)=|A_1|\le3$.
\end{proof}

\begin{lemma}\label{lem4}
Suppose $i\in[-k_1,k_2]^{*}$ and $j\in[-k_2,k_1]^{*}$.  
Then we have 
\begin{itemize}
\item [(a)] $e\notin S(i, j)$,
\item [(b)] $|S(i,j)|=n^2-n-\psi(j-i,-i,j)$ if $0<j-i\le k_1+k_2$ or $\gcd(j-i, |G|)=1$. 
\item [(c)] $|S(i,j)|\ge n^2-n-\frac{(C-1)n}{2}-\psi(k_1+k_2+1,-i,j)$ if $j-i=k_1+k_2+1$.
\end{itemize}
Here $C$ is the one defined in Lemma \ref{lem8}. 
\end{lemma}
\begin{proof} Suppose $g^i h^j=e$ with $g\neq h$.  As $T^{(i)}\cap T^{(-j)}=\emptyset$ unless $i=-j$,  it follows that $g^i h^{-i}=e$. But then $g^i=h^i$. Since $|T^{(i)}|=n$, then $g=h$. This is impossible and (a) follows. 

To compute $|S(i,j)|$, we need to count the number of distinct pairs $(g_1, g_2), (g_3, g_4)$ with 
$g_1^ig_2^j=g_3^ig_4^j\in S(i,j)$ such that $g_1,  g_2,g_3, g_4\in T$, $g_1\ne g_2$ and $g_3\ne g_4$. Then 
\begin{equation} \label{eq1}
	g_3^{-i}g_2^j=g_1^{-i}g_4^j. 
\end{equation}

If $g_1\neq g_4$, then $g_3^{-i}g_2^j=g_1^{-i}g_4^j\in S(-i, j)$. 
Since $g_1\neq g_2$, either $g_2=g_3$ or $g_1=g_3$. 
In the first case, 
$g_2^{j-i}\in S(-i, j)$. Note that each element of the form $g^{j-i}\in T^{(j-i)}\cap S(-i, j)$ gives one set of solution. Therefore, the number of pair of solutions in Equation 
~(\ref{eq1}) is exactly $\psi(j-i,-i,j)$ and our desired result follows. 
If $g_1=g_3$, then $(g_1, g_2)=(g_3, g_4)$, which contradicts our assumption. 

In case $g_1=g_4$, we use the same argument as before if $g_2\neq g_3$. But then, the condition obtained would be the same as before. Therefore, this case does not give us more solutions.

Finally, we consider the case $g_1=g_4$ and $g_2=g_3$. 
We thus have $g_1^{j-i}=g_2^{j-i}$ with $g_1\neq g_2$. 
As $k_1+k_2 \geq j-i>0$, then it follows from Lemma \ref{lem2} that $g_1=g_2$. Again, this is impossible. If $\gcd(j-i, |G|)=1$, then it also means $g_1=g_2$. We have thus proved (b). 

For (c),  let $C$ be as defined in Lemma \ref{lem8}. Therefore, there are at most
 \[
 \binom{C}{2}\frac{n}{C}=\frac{(C-1)n}{2}
 \]
 such solutions. Thus, we conclude that  $|S(i,j)|\ge n^2-n-\frac{(C-1)n}{2}-\psi(k_1+1,-i,j).$
\end{proof}

\begin{lemma}\label{lemma-new1}
For any $i\in[-k_1,k_2]^{*}$ and $j,\ell \in[-k_2,k_1]^{*}$, 
	\[|S(i,j)\cap T^{(\ell)}|=\left\{ \begin{array}{ll}
\psi(\ell -i,j) & \mbox{ if } \ell-i \geq k_1+k_2+1, \\
0 & \mbox{ if } \ell -i \leq k_1+k_2. \end{array}\right.\]
\end{lemma}
\begin{proof}
Consider any element $g_1^{i}g_2^{j}=g_3^\ell\in S(i,j)\cap T^{(\ell)}$.  
Rearranging, we obtain $g_2^{j}=g_1^{-i}g_3^{\ell}$. 
 If $g_1\ne g_3$, then we have $g_1^{-i}g_3^{\ell}\in S(-i,\ell)$. Since $-i,j,\ell\in [-k_2,k_1]^{*}$, this is impossible in view of Lemma~\ref{lem1}. 

If $g_1=g_3$, then we obtain $g_2^{j}=g_1^{\ell-i}$. If $\ell-i\leq k_1+k_2$, then in view of Lemma \ref{lem2} (b), we see that $\ell-i=j$ and $g_2=g_1$. This again is impossible. We may then assume 
$\ell-i \geq k_1+k_2+1$. In that case, the number of solutions for the equation $g_1^{i}g_2^{j}=g_3^\ell$ is precisely $\psi(\ell -i,j)$. 
\end{proof}

\begin{lemma}\label{lem3}
	Let $i_1,i_2\in[-k_1,k_2]^{*}$ and $j_1,j_2\in[-k_2,k_1]^{*}$ with $(i_1,j_1)\ne(i_2,j_2)$. Suppose  $j_1-i_2\ge j_2-i_1$ and $j_2-i_1\in[-2k_2,k_1+k_2]^{*}$. Then, we have
\[  |S(i_1,j_1)\cap S(i_2,j_2)|\leq \left\{\begin{array}{ll}
\psi(j_1-i_2,-i_1,j_2)+\psi(j_2-i_1,-i_2,j_1)\\
+\psi(j_1-i_2,j_2-i_1) & \mbox{if } 
j_1-i_2>k_1+k_2\ge j_2-i_1 \neq 0, \\
\psi(j_1-i_2,-i_1,j_2)+\psi(j_2-i_1,-i_2,j_1)  & \mbox{if } k_1+k_2\geq j_1-i_2\geq j_2-i_1\neq 0,  \\
\psi(j_1-i_2,-i_1,i_1)+(n-1)\psi(j_1-i_2,0) & \mbox{if } j_2-i_1=0.  
\end{array} \right. \]
\end{lemma}
\begin{proof}
	Consider  elements $g_1^{i_1}g_2^{j_1}\in S(i_1,j_1)$ and $g_3^{i_2}g_4^{j_2}\in S(i_2,j_2)$ with $g_1\ne g_2$ and $g_3\ne g_4$. 
We need to find such $g_1, g_2, g_3, g_4$ with
\begin{equation} \label{E1}
g_1^{i_1}g_2^{j_1}=g_3^{i_2}g_4^{j_2}.\end{equation}
	Rearranging, we obtain
	\begin{align*}
		g_3^{-i_2}g_2^{j_1}=g_1^{-i_1}g_4^{j_2}.
	\end{align*}

{\bf{Case (1): $g_3\ne g_2$ and $g_1\ne g_4$.}}

For this case, we have 
\[ g_3^{-i_2}g_2^{j_1}=g_1^{-i_1}g_4^{j_2}\in S(-i_2,j_1)\cap S(-i_1,j_2). \]
Since $-i_1,-i_2,j_1,j_2\in [-k_2,k_1]^{*}$, it follows from Theorem \ref{main} (2) that $(-i_2,j_1)=(-i_1,j_2)$ or 
$(-i_2,j_1)=(j_2, -i_1)$. If $j_1=j_2$, then $(-i_1,j_1)\ne(-i_2,j_2)$. Hence,  $j_1\neq j_2$ and $(-i_2,j_1)\ne(j_2,-i_1)$. But then $g_3=g_4$ and $g_1=g_2$, which is again impossible. 

{\bf{Case (2): $g_3=g_2$ and $g_1\ne g_4$.}}

For this case, we obtain $g_2^{j_1-i_2}=g_1^{-i_1}g_4^{j_2}$. Since $g_1^{-i_1}g_4^{j_2}\in S(-i_1,j_2)$ and $-i_1,j_2\in[-k_2,k_1]^{*}$, there are at most $\psi(j_1-i_2,-i_1,j_2)$ solutions in this case. Note that we do not include those solutions with $g_1=g_2$, so we only get a bound in term of $\psi(j_1-i_2,-i_1,j_2)$.
  
{\bf{Case (3): $g_4=g_1$ and $g_2\ne g_3$.}}

  By using a similar argument as in Case (2),  we obtain $\psi(j_2-i_1,-i_2,j_1)$ solutions.
	
{\bf{Case (4):  $g_3=g_2$ and $g_1=g_4$.}}	

For this case, we get $g_2^{j_1-i_2}=g_1^{j_2-i_1}$. 
Note that $-2k_2\leq j_2-i_1\le 2k_1$ and we may assume $j_1-i_2\geq j_2-i_1$. 
If $k_1+k_2\geq j_1-i_2$, then there exist $a,b\in[-k_2,k_2]^{*}$ such that $j_1-i_2-a,j_2-i_1-b\in[-k_2,k_1]^{*}$. This leads to $g_2^{j_1-i_2-a}g_1^{-b}=g_1^{j_2-i_1-b}g_2^{-a}$. Hence $g_1=g_2$. Therefore, Equation (\ref{E1}) has no solution with $g_1\neq g_2$.

In case $j_1-i_2>j_2-i_1 \neq 0$, for each $g_2\in T$, there exists at most one solution for $g_1$ as 
$|T^{(j_2-i_1)}|=n$. 
Thus, 
	\[|S(i_1,j_1)\cap S(i_2,j_2)|= \psi(j_1-i_2,-i_1,j_2)+\psi(j_2-i_1,-i_2,j_1)+\psi(j_1-i_2,j_2-i_1).\]
However, if $j_2=i_1$, then $g_1^{-i_1}g_4^{j_2}=e$ whenever $g_1=g_4$. For each $g_1$ chosen, $g_2^{j_1-i_2}=e$. Therefore, the number of choices of $g_2$ is exactly $\psi(j_1-i_2, 0)$. In other words, those $g_1^{i_1}g_2^{j_1}$ are counted $n$ times instead of one time. 
\end{proof}

\medskip

In order to compute $|S(i,j)|$ with $j-i=k_1+k_2+1$, we need to compute $\psi(k_1+k_2+1, 0)$. This is really the reason why we need to assume $k_1+k_2+1$ is composite. 

\medskip

\section{$\mathcal{B}(n,2,3,0)$}\label{sub3}

In Section 5,  we find a bound of $n$ in terms of $k_1$ and $k_2$ such that there is no such tiling whenever $n$ exceeds the bound. However, in dealing with specific cases, the general bound might not be good enough to solve those cases completely. Very often, we need to vary the index sets so as to find a better bound for $n$.  We can then deal with those $n$ less than the bound found. Here, we illustrate how it works in case of $B(n, 2, 3, 0)$, which also illustrates the simple idea behind our proofs. 

Let \[\mathcal{X}=\{(i,j): i\in[-3,-1], j\in[1,3], j-i\le 5\}.\]
In this case, $k_1=3, k_2=0$ and the smallest prime that divides $k_1+k_2+1$ is $2$. 
Our objective is to derive an inequality using 
\[ |G|-1\geq |\cup_{(i,j)\in\mathcal{X}}S(i,j)|. \]
We need to find $|S(i_1,j_1)|$ and $|S(i_1,j_1) \cap S(i_2,j_2)|$ for 
any $(i_1,j_1), (i_2,j_2)\in \mathcal{X}$.

\begin{lemma}\label{lemma-new5}
	\begin{enumerate}
		\item[(a)] If $(i,j)\in\mathcal{X}$ with $j-i\le 3$, then $|S(i,j)|=n^2-n$.
		\item[(b)] If $(i,j)\in\mathcal{X}$ with $j-i= 4$, then $|S(i,j)|=n^2-n-\frac{(C-1)n}{2}-\psi(4,-i,j)$.
		\item[(c)] If $(i,j)\in\mathcal{X}$ with $j-i= 5$, then $|S(i,j)|=n^2-n-\psi(5,-i,j)$.
	\end{enumerate}
Here $C$ is the one we defined in Lemma \ref{lem8}.
\end{lemma}
\begin{proof}
	(a), (b) and (c) follow directly from Lemma \ref{lem4} (b) as 
$\gcd(\frac{9n^2}{2}-\frac{3n}{2}+1,5)=1$ implies $\psi(5, 0)=0$.
\end{proof}

By Lemma \ref{lemma-new5},
\begin{align*}
	&\sum_{(i,j)\in\mathcal{X}}|S(i,j)|\\
	=&3(n^2-n)+3(n^2-(C+1)n/2)+2(n^2-n) -\psi(4,1,3)-\psi(4,2,2)-\psi(4,1,3)-\\
	&\psi(5,2,3) -\psi(5,2,3).
\end{align*}

To find $|S(i_1,j_1)\cap S(i_2, j_2)|$, we apply Lemma \ref{lem3}. 
There are $8$ elements in the index set $\mathcal{X}$ and we need to determine ${8 \choose 2}$ pairs of $|S(i_1, j_1)\cap S(i_2, j_2)|$. Note that $2\leq j_2-i_1\leq 6$. To avoid double counting, we may assume $\ell=j_2-i_1\geq j_1-i_2$. 
The argument is straightforward but tedious. We just illustrate how to group relevant terms together to obtain our bound. 
For example, as $\ell\leq 6$, we consider intersection of
$|S(-3,\alpha)\cap S(\beta, 3)|$ and vary $\alpha$ and $\beta$. We then group all terms of the form $\psi(\ell, i, j), \psi(\ell, i)$ together. Note that in this case $\alpha=1,2$ and $\beta=-1,-2$. 
Then, we consider the case $\ell=5$, in which we have two cases, $S(-2, \alpha)\cap S(\beta, 3), S(-3, \alpha)\cap S(\beta, 2)$. Altogether, we get 
\[ S_1=\psi(6, 2, 2)+2\psi(6, 1,2)+\psi(6, 1,1)+\psi(6, 4)+ 2\psi(6, 3)+\psi(6,2).\]
\[ S_2=2\psi(5, 2,3)+2\psi(5, 2,2)+2\psi(5, 1,3) +4\psi(5, 1,2)+2\psi(5,1,1)+4\psi(5, 4)+4\psi(5, 3)+2\psi(5, 2).\]
\[ S_3=3\psi(4, 2,2)+4\psi(2,3)+ \psi(4, 3,3)+2\psi(4, 1, 3)+ 6\psi(4, 1,2)+3\psi(4,1,1)+6\psi(4, 3)+3\psi(4,2)+2\psi(4,4).\]
We only need to consider $\ell \geq 4$ as for $\ell \leq 3$, the term $\psi(\ell, i, j),\psi(\ell, i)=0$ when $i\leq j$ and $i, j\in [1,3]$ by Lemma \ref{lem1}.
Therefore, 
\begin{align*}
 |\cup_{(i,j)\in\mathcal{X}}S(i,j)| \geq & 3(n^2-n)+3(n^2-5n/2)+2(n^2-n) -\psi(4,1,3)-\psi(4,2,2)-\\
 & \psi(4,1,3)-\psi(5,2,3) -\psi(5,2,3)-S_1-S_2-S_3.
\end{align*}
For each $\ell$, combine all the terms of the form $\psi(\ell, i, j)$ and $\psi(\ell, i)$ for any $\ell$ with 
$i, j\in [1,3]$ as one sum. 
For example, the sum 
\begin{align*}
&\psi(5,2,3) +\psi(5,2,3)+ 2\psi(5, 2,3)+2\psi(5, 2,2)+2\psi(5, 1,3) +4\psi(5, 1,2)+2\psi(5,1,1)+\\
&4\psi(5, 4)+4\psi(5, 3)+2\psi(5, 2)\leq 4n
\end{align*}
as there are four $\psi(5,2,3) $ in the sum. Note that $\psi(5,2,3)=\psi(5, 3,2)$. Using a similar argument,
\[ \psi(6, 2, 2)+2\psi(6, 1,2)+\psi(6, 1,1)+\psi(6, 4)+ 2\psi(6, 3)+\psi(6,2) \leq 2 n  \]
and
\begin{align*}
	&\psi(4,1,3)+\psi(4,2,2)+\psi(4,1,3)+3\psi(4, 2,2)+4\psi(2,3)+ \psi(4, 3,3)+2\psi(4, 1, 3)+ 6\psi(4, 1,2)+\\
	&+3\psi(4,1,1)+6\psi(4, 3)+3\psi(4,2)\leq 6n.
\end{align*}
For the terms of the form $\psi(i, j)$ with $6\geq i\neq j\geq 4$, we just bound $\psi(i,j)$ by $n$. Altogether, we have $5$ such terms and the bound is $5n$. Lastly, to find $\psi(4, 4)$, we apply previous argument to obtain $\psi(4,4)\leq 2(C-1)n$. Combining the above inequalities, we then get
\[ |G|-1\geq |\cup_{(i,j)\in\mathcal{X}}S(i,j)| \geq 8n^2-5n-17n-(C-1)n-\frac{3(C+1)n}{2}.\]
Hence, we have
\begin{align}
	7n\le42+5C.\label{ineq1}
\end{align}

By Lemma \ref{lem8}, $C\leq 4$. Hence, we conclude $n\leq 8$. For $n=7$ and $8$, $|G|$ is odd. As the smallest prime that divides $k+1=4$ is $2$, $C=1$ and inequality (9) fails. Hence, $n\neq 7$ or $8$. For $3\leq n\leq 6$, we use a computational search 
and check that no tiling of $\mathbb{Z}^n$ exists by $B(n, 2,3,0)$ if $n=4, 5$ and $6$. 
In case $n=3$, we take $T=\{1,10,26\}$ in the group $\mathbb{Z}_{37}$ and the conditions in Theorem \ref{main} are satisfied. This 
verifies that a lattice tiling of $\mathbb{Z}^{3}$ by $B(3,2,3,0)$ exists.

\section{$k_1=k\ge2$ and $k_2=k-1$}\label{sub2}

In this section, we illustrate how our idea be used to solve completely for an infinite families. Let $k_1=k$ and $k_2=k-1$. We define 
\[ \mathcal{Y}=\{(i,j): i\in[-k+1,k-1]^{*}, \ j\in[-k+1,k]^{*}, j\ge i\}.\]

Recall that we may assume two conditions in Theorem \ref{main} hold. Thus, we conclude
\[G=e+\sum_{i\in [-k+1, k]^*}T^{(i)}+S(k, k) +\sum_{(i,j) \in  \mathcal{Y}}
S(i,j)=e+\sum_{i\in [-k+1, k]^*}T^{(-i)}+S(-k, -k)+\sum_{(i,j) \in \mathcal{Y}}
S(-i,-j).\]

By applying these two group ring equations,  we deduce the following:

\begin{lemma}\label{lemma19}
\begin{itemize}
\item [(a)] $|S(-k,1)|=|S(-k, 1)|=n^2-n$; and $|S(-k, 2)|=n^2-n$ if $k\geq 3$. 
\item [(b)] $S(-k,1), S(-k, -1), \{e\} \cup \bigcup_{j\in [-k+1, k-1]^*}T^{(i)}\cup 
\bigcup_{i,j \in  [-k+1, k-1]^*,i\le j}
S(i,j)$ are pairwise disjoint. 
\item [(c)] $S(-k,1), S(-k, -1), S(-k, 2), \{e\} \cup \bigcup_{j\in [-k+1, k-1]^*}T^{(i)}\cup 
\bigcup_{i,j \in  [-k+1, k-1]^*,i\le j}
S(i,j)$ are pairwise disjoint if $k\geq 3$. 
\end{itemize}
\end{lemma}
\begin{proof}
Note that (a) follows from the equalities $|S(-k,i)|=|S(-i ,k)|=n^2-n$. (b) and (c) follow from above equality.
\end{proof}

\medskip

In view of Lemma \ref{lemma19}, we see that 
\[ |(S(-k, 1)\cup S(-k, -1)|=2(n^2-n) \mbox{ and } \]
\[ |\{e\} \cup \bigcup_{i\in [-k+1, k]^* }T^{(i)}\cup \bigcup_{(i,j) \in  \mathcal{Y}} S(i,j)|=|G|- |S(k, k)|=|G|- \frac{n(n-1)}{2}.\]

\medskip

Next, we study  the intersection between $S(-k,u)$ and $S(j,k)$ for $u=-1,1, 2$ and $j\in [-k+1,k-1]$. By Lemma \ref{lem3}, we obtain 
\[|(S(-k,u) \cap(T^{(k)}\cup \bigcup_{j\in  [-k+1, k-1]^* }S(j,k))| = \psi(2k, u)+\psi(2k,0)(n-1) +X, \mbox{ where } \]
\[  X= \sum_{j\in  [-k+1, k-1]^*  ,j\ne u} \psi(2k,-j,u)   +\psi(u-j,k,k)+\psi(2k,u-j). \]
Note that $ \psi(u-j,k,k)$ is nonzero only if $u-j\notin  [-k+1, k]^*$. By Lemma \ref{lem1} again, we see that 
\[ X + \psi(2k, u)+\psi(2k, 0)\leq \left\{ \begin{array}{ll}
n & \mbox{ if } u=1, \\
n+\psi(k+1, k, k) &  \mbox{ if } u=2, \\
n+\psi(-k, k, k)  & \mbox{ if } u=-1. \end{array}\right. \]

We are now ready to prove Theorem \ref{mainthm5}. Let
\[ G_1=S(-k, 1)\cup S(-k, -1) \cup \{e\} \cup \bigcup_{i\in  [-k+1, k]^* }T^{(i)}\cup \bigcup_{(i,j) \in  \mathcal{Y}}S(i,j).\]
Clearly, we have 
\[|G|\geq |G_1|\geq 2(n^2-n)+|G|- \frac{n(n-1)}{2}-(2+2 \psi(2k, 0))n+4\psi(2k,0)-\psi(-k, k, k).\]
We then deduce that 
\begin{equation}\label{eq10}
 0 \geq \frac{3n(n-1)}{2}-(2+2 \psi(2k, 0))n+4\psi(2k,0)-\psi(-k, k, k) .\end{equation}
By Lemma \ref{lem8} (b), $\psi(2k, 0)\leq 3$. 
Since $\psi(k+1, k,k)\leq n$. We see  inequality (\ref{eq10}) fails unless $n\leq 6$.
However, if $n=4,5,6$, $\psi(2k,0)\leq 1$ and inequality (\ref{eq10}) fails again. 
Thus $n=3$. If $k=2$,
a computational search confirms that no lattice tiling of $\mathbb{Z}^{3}$ by $B(3,2,2,1)$ exists.

Next, we consider the case $k\geq 3$ and 
\[ G_2=(S(-k, 1)\cup S(-k, 2) \cup \{e\} \cup \bigcup_{i\in  [-k+1, k]^* }T^{(i)}\cup \bigcup_{(i,j) \in  \mathcal{Y}}S(i,j).\]
Note that $|G|$ is odd when $n=3$ and thus $\psi(2k,0)=0$. 
Replacing $\psi(-k, k, k)$ by $\psi(k+1, k, k)$ for the inequality (\ref{eq10}), we obtain 
\[  0 \geq 3\times 3-2\times3 -\psi(k+1, k, k) .\]
This is possible only if $3=\psi(k+1, k, k)$. 
By Lemma \ref{lem2}, $|T^{(k+1)}|=3$ and $|S(k, k)|=3$. Suppose $T=\{g_1, g_2, g_3\}$.  Therefore, $\{g_1^{k+1},  g_2^{k+1},  g_3^{k+1}\}=\{
 g_1^{k}g_2^{k}, g_2^{k}g_3^{k}, g_3^{k}g_1^{k}\}$. 
 We then obtain 
\[ g_1^{k+1}=g_2^k g_3^k, \ g_2^{k+1}=g_1^k g_3^k\mbox{ and } g_3^{k+1}=g_1^k g_2^k.\]
(Note that if $g_1^{k+1}=g_1^k g_i^k$, we immediately get a contradiction.)
Then, 
\[ g_1^{k+1}g_2=g_2^{k+1} g_3^k=g_1^k g_3^{2k}=g_1^{2k}g_3^{k-1}g_2^k.\]
Thus, we obtain $g_1^{-k+1}=g_3^{k-1} g_2^{k-1}$, which is impossible.

\section{General cases}\label{gen}

In this section, our objective is to find a bound $N(k_1, k_2)$ such that for any $B(n, 2, k_1, k_2)$ with $k_1>k_2$ and $k_1+k_2+1$ not a prime,  then no lattice tiling of $\mathbb{Z}^n$ by $\mathcal{B}(n,2,k_1,k_2)$ exists whenever $n\geq N(k_1, k_2)$,

In the previous section, we deal with $\mathcal{B}(2, n, k, k-1)$ and clearly $k+k-1+1$ is  composite. As illustrated in the proof, one of the key calculation is to find a bound for $\psi(2k, 0)$. It turns out that for $m\in [k_1+k_2+1, 2k_1]$, 
we can obtain a good bound to compute for $\psi(m, 0)$ if $m$ is composite.
In case $k_1+k_2+1$ is prime, it is possible that $\psi(k_1+k_2+1,0)=n$. In that case, $G$ is $p$-elementary, and for any $i,j$, $S(i, j)$ is actually one of those $S(u, v)$ with $u,v\in [-k_2, k_1]^*$.
Our method will then fail. However, it is still possible that even if $k_1+k_2+1$ is a prime, $G$ need not be $p$-elementary. (Indeed, we use a program to test out the case when $k_1+k_2+1=5$, except for $n=2$, we cannot find any other prime power for $|G|$.) In that case, it is still possible to obtain result similar to Theorem~\ref{mainthm1}. Since the proofs for those cases are more complicated and not shedding new light, we just deal with the case $k_1+k_2+1$ not a prime. 

\medskip

 Let $M\in[1,k_1-k_2]$ such that all elements in $[k_1+k_2+1,k_1+k_2+M]$ are composite. Define the following index sets:
\begin{align*}
	&\mathcal{X}_1=\{(i,j): i\in[-k_1,-k_2-1],\ j\in[k_2+1,k_1], j-i\le k_1+k_2+1\},\\
	&\mathcal{X}_2=\{(i,j): i\in[-k_2-M,-k_2-1],\ j\in[-k_2,k_2]^{*}\},\\
  & \mathcal{X}=\mathcal{X}_1\cup \mathcal{X}_2, \\
	&\mathcal{Y}=\{(i,j): i\in[-k_2,k_2]^{*},\ j\in[-k_2,k_1]^{*}, j\ge i\},\\
&\mathcal{Z}=\{(i,j): i\in[k_2+1,k_1]^{*},\ j\in[k_2+1,k_1]^{*}, j\ge i\}.
\end{align*}

Recall that
\[ G=\{e\}\cup \bigcup_{i\in [-k_2, k_1]^*} T^{(i)} \cup \bigcup_{(i,j)\in \mathcal{Y}} S(i,j) \cup \bigcup_{(i,j)\in \mathcal{Z}} S(i,j).\]

As the computations involved in this section are rather complicated, we first illustrate the simple idea behind the proof. To prove Theorem~\ref{mainthm1}, we aim to obtain a contradiction when replacing the index set  $ \mathcal{Z}$ by $\mathcal{X}$. We consider the set 
\[ G'=\{e\}\cup \bigcup_{i\in [-k_2, k_1]^*} T^{(i)} \cup \bigcup_{(i,j)\in \mathcal{Y}} S(i,j) \cup \bigcup_{(i,j)\in \mathcal{X}} S(i,j).\]
Our objective is to find a condition for which $|G'|\leq |G|$.

Note that by assumption, we have
\begin{equation}  |\{e\}\cup \bigcup_{i\in [-k_2, k_1]^*} T^{(i)} \cup \bigcup_{(i,j)\in \mathcal{Y}} S(i, j)|=|G|-\sum_{k_2+1\le i\le j\le k_1}|S(i,j)|=|G|-\frac{(k_1-k_2)^2}{2}(n^2-n).\end{equation}
To compute $|G'|$, we need to compute
\[ \left|\bigcup_{(i,j)\in \mathcal{X}} S(i,j)\right| \mbox{ and } 
\left|\left(\bigcup_{(i,j)\in \mathcal{X}} S(i,j) \right) \cap 
\left( \bigcup_{i\in [-k_2, k_1]^*} T^{(i)} \cup \bigcup_{(i,j)\in \mathcal{Y}} S(i,j)\right)\right|.\]

As we show later, $|S(i, j)|$ is roughly the order of $n^2$ for any $(i, j) \in \mathcal{X}$; 
and the number of elements in $\mathcal{X}$ is larger than $\mathcal{Z}$. To get a contradiction for large enough $n$, we only need to show that there exists a constant $C(k_1, k_2)$ such that $|S(i,j)\cap S(i', j')|\leq C(k_1, k_2)n $ 
for any distinct $(i, j), (i', j')\in \mathcal{X}\cup \mathcal{Y}$. Since 
$|\mathcal{X}\cup \mathcal{Y}|$ is roughly $(k_1+k_2)^2$, 
$\sum_{(i, j), (i', j')\in \mathcal{X}\cup \mathcal{Y}}|S(i,j)\cap S(i', j')|$ is bounded by 
$2(k_1+k_2)^2 C(k_1, k_2) n$. But then as
$\sum_{(i,j)\in \mathcal{X}} |S(i, j)|- \sum_{(i,j)\in \mathcal{Z}} |S(i, j)|\geq An^2-Bn $ for some constants $A, B$ that depend only on $k_1, k_2$, we see that no such tiling exists if $n$ is sufficiently large. To prove this weaker version, we can avoid complicated computations in our proofs shown below. However, to get a better understanding on the bound, and to solve some specific cases, we need to find a better bound so as to resolve those case completely.
The following is an immediate consequence of Lemma \ref{lem4} and $|S(i,j)|=|S(-i,-j)|$.

\begin{lemma}\label{lemma-new4}
	\begin{enumerate}
		\item [(a)] If $(i,j)\in\mathcal{X}_1$ and $j-i\le k_1+k_2$, then $|S(i,j)|=n^2-n-\psi(j-i,-i,j).$
		\item [(b)] If $(i,j)\in\mathcal{X}_1$ and $j-i= k_1+k_2+1$, then $|S(i,j)|\ge n^2-n-\frac{(C-1)n}{2}-\psi(k_1+k_2+1,-i,j)$ where $C$ is the number defined in Lemma \ref{lem8} (a). 
		\item [(c)] If $(i,j)\in\mathcal{X}_2$, then $|S(i,j)|=n^2-n$.
		\item [(d)] For any $(i,j)\in\mathcal{X}$, the identity element	$e$ is not contained in $S(i,j)$.\end{enumerate}

\end{lemma}

\medskip

By Lemma \ref{lemma-new4}, we deduce the following: 

\begin{lemma}\label{la}
\[\sum_{(i,j)\in\mathcal{X}}|S(i,j)|\ge(\frac{(k_1-k_2+1)(k_1-k_2)}{2}+2Mk_2)(n^2-n)-(k_1-k_2)\frac{(C-1)n}{2}-\sum_{(i,j)\in\mathcal{X}_1}\psi(j-i,-i,j). \]
\end{lemma}

To find a bound for $|G'|$, take note that sets in  $\cup_{(i,j)\in \mathcal{Y}} S(i, j)$ are disjoint by assumption. 
Our next step is to compute 
\[\frac{1}{2}\sum_{(i_1,j_1)\in \mathcal{X}} \sum_{(i_1,j_1)\neq (i_2,j_2)\in \mathcal{X}} |S(i_1, j_1)\cap S(i_2, j_2)|,  \sum_{(i_1,j_1)\in \mathcal{X}} \sum_{(i_2,j_2)\in \mathcal{Y}} |S(i_1, j_1)\cap S(i_2, j_2)|
\mbox{ and } \]
\[ \sum_{(i_1,j_1)\in \mathcal{X}}  \sum_{\ell\in [-k_2, k_1]^* }|S(i_1, j_1)\cap T^{(\ell)}|.\]
Note that in the first sum above, we take only one half to avoid double counting of $|S(i_1, j_1)\cap S(i_2, j_2)|$. 
For $(i_1, j_1),  (i_2, j_2)\in \mathcal{X} \cup \mathcal{Y}$, we apply Lemma \ref{lem3} and obtain
\[ |S(i_1,j_1)\cap S(i_2,j_2)|\leq \left\{\begin{array}{ll}
\psi(j_1-i_2,-i_1,j_2)+\psi(j_2-i_1,-i_2,j_1)+\psi(j_2-i_1,j_1-i_2) \mbox{ if } j_1\neq i_2 \\
\psi(j_2-i_1,-i_2,j_1)+(n-1) \psi(j_2-i_1,0) \mbox{ if } j_1=i_2 
\end{array} \right. \]
and 
\[ |S(i_1,j_1)\cap T^{(\ell)}|= \psi(\ell-i_1, j_1).\]

Before we continue, we recall some results which we apply frequently in this section. 

\begin{lemma} \label{ex1}
For any $\ell$, 
\begin{itemize}
\item [(a)] $\sum_{i\in [-k_2, k_1]^*} \psi(\ell, i)+ \sum_{(i, j)\in \mathcal{Y }\cup \mathcal{Z}}  \psi(\ell,  i, j)\leq n $; and 
$\sum_{(i, j)\in \mathcal{Y }\cup \mathcal{Z}}  \psi(\ell,  i, j)=0$ if $\ell \in [-k_2, k_1]^*$. 
\item [(b)] $ \psi(\ell, i)=0$ if $i\in [-k_1, 0]$ and $0<\ell-i \leq k_1+k_2$. 
\item [(c)] $\sum_{i=1}^{k_1+k_2} \psi(\ell, u+i)\leq n$ if $u\in [-k_1, 0]$. 
\end{itemize}
\end{lemma}

\begin{lemma}\label{m1}
\[ S_1=\sum_{(i, j)\in \mathcal{X}}  \sum_{u \in [-k_2, k_1]^*}|S(i, j)\cap T^{(u)}|\leq \sum_{\ell=k_1+k_2+1}^{2k}
\sum_{j \in [k_2+1, k_1]} 2 c_1(\ell) \psi(\ell, j).\]
Here $c_1(\ell)=|\{ (-i, j)\in [k_2+1, k_1]\times [k_2+1, k_1]: j-i =\ell\}|$. 
\end{lemma}
\begin{proof} For any $u\in [-k_2, k_1]^*$,
 $ |S(i,j)\cap T^{(u)}|= \psi(u-i, j)$.
Note that  $(i, j)\in \mathcal{X}$ and $j\in [k_2, k_1]^*$. Now, set $\ell=u-i$. 
By Lemma \ref{ex1} (b), $\psi(u-i, j)\neq 0$ only if $\ell =u-i\geq k_1+k_2+1$. 
As $k_2+1\leq -i\leq k_1$, $\ell\geq k_1+k_2+1$ only if $u\geq k_2+1$. 
Therefore, 
\[ \sum_{(i, j)\in \mathcal{X}_1}  \sum_{u \in [-k_2, k_1]^*}|S(i, j)\cap T^{(u)}|\leq \sum_{\ell=k_1+k_2+1}^{2k}
\sum_{j \in [k_2+1, k_2]^*} c_1(\ell) \psi(\ell, j).\]
Clearly, 
\[ \sum_{(i, j)\in \mathcal{X}_2}  \sum_{u \in [-k_2, k_1]^*}|S(i, j)\cap T^{(u)}|\leq  \sum_{(i, j)\in \mathcal{X}_1}  \sum_{u \in [-k_2, k_1]^*}|S(i, j)\cap T^{(u)}|\]
and our desired result follows. 
\end{proof}

\medskip

\begin{lemma} \label{l1}
	\begin{align*}
		S_2=&\frac{1}{2}\sum_{(i_1, j_1)\neq (i_2, j_2)\in \mathcal{X}_1} 
		\psi(j_1-i_2,-i_1,j_2)+\psi(j_2-i_1,-i_2,j_1)+
		\sum_{(i_1, j_1)\in \mathcal{X}_1} \psi(j_1-i_1,  -i_1, j_1)\\
		\leq&  \sum_{\ell=k_1+1}^{2k_1}  \sum_{(i,j)\in  [k_2+1, k_1]^2} 2c_1(\ell) \psi(\ell,  i, j).
	\end{align*}
Here $c_1(\ell)$ is as defined in Lemma \ref{m1} and 
$\sum_{\ell=k_1+1}^{2k_1} c_1(\ell)\leq (k_1-k_2)^2$.  
\end{lemma}
\begin{proof} 
We set $\ell=j_1-i_2$. Note that $(-i_1, j_2), (-i_2, j_1)$ and $(-i_1, j_1)$ all lie in 
$\mathcal{Y}\cup \mathcal{Z}$. It then follows from 
Lemma \ref{ex1} (a), we only need to consider terms with $\ell \geq k_1+1$. 
Since $(-i_1, j_1), (-i_1, j_2), (-i_2, j_1)\in [k_2+1, k_1]$, 
and the number of pairs $(j_1, i_2)$ (resp. $(j_2, i_1), (j_1, i_1))$ with $j_1-i_2=\ell$  is at most  $c_1(\ell)$. We thus get our desired bound. 
Clearly, $\sum_{\ell=k_1+1}^{2k_1} c_1(\ell)\leq (k_1-k_2)^2$ as there are only $k_1-k_2$ choices for both $j_1$ and $i_2$. 
\end{proof}

\medskip

Using a similar argument as before, we obtain the following:

\begin{lemma} \label{l2}
	\begin{align*}
		S_3=&\sum_{(i_1, j_1)\in \mathcal{X}_1,  (i_2, j_2)\in \mathcal{X}_2}
		\psi(j_1-i_2,-i_1,j_2)+\psi(j_2-i_1,-i_2,j_1)\\
		\leq&   \sum_{\ell=k_1+1}^{2k_1}\sum_{(-i_1,j_2)\in  [k_2+1, k_1]\times [-k_2, k_2]^*}  c_2(\ell)\psi(\ell,  -i_1, j_2)+ 
		\sum_{\ell=k_1+1}^{2k_1}\sum_{(-i_2,j_1)\in  [k_2+1, k_2+M]\times [k_2+1, k_1]}c_3(\ell) \psi(\ell,  -i_2, j_1).
	\end{align*}
Here $c_2(\ell)=|\{ (-i_2,j_1) \in [k_2+1, k_2+M]\times [k_2+1, k_1] : j_1+i_2=\ell \}|$ and  $c_3(\ell)=|\{ (-i_1,j_2) \in [k_2+1, k_1]\times [-k_2, k_2]^* :j_2-i_1=\ell \}|$. Furthermore, 
\[\sum_{\ell=k_1+1}^{2k_1} c_2(\ell)\leq M(k_1-k_2)  
 \mbox{ and } \sum_{\ell=k_1+1}^{2k_1}   c_3(\ell)\leq (k_1-k_2)k_2.\]
\end{lemma}

\begin{lemma} \label{l3}
	\begin{align*}
		S_4=&\sum_{(i_1, j_1)\in \mathcal{X}_1,  (i_2, j_2)\in \mathcal{Y}}
		\psi(j_1-i_2,-i_1,j_2)+\psi(j_2-i_1,-i_2,j_1) \\
		\leq&\sum_{\ell=k_1+1}^{2k_1}  \sum_{(-i_1,j_2)\in  [k_2+1, k_1] \times [-k_2, k_1]^*} c_3(\ell)\psi(\ell,  -i_1, j_2)+ 
		\sum_{\ell=k_1+1}^{2k_1} \sum_{(i_2,j_1)\in   [-k_2, k_2]^* \times [k_2+1, k_1]} (c_1(\ell)+c_3(\ell)) \psi(\ell,  -i_2, j_1) .
	\end{align*}
Here  $c_3(\ell)$ is defined as in  Lemma \ref{l2}.
\end{lemma}

\medskip

The proof of Lemma \ref{l3} is similar to the proofs shown before. Take note that for $(i_1, j_1)\in \mathcal{X}_1,  (i_2, j_2)\in \mathcal{Y}$,
$j_2-i_1\geq 1$. Again, we set $\ell=j_2-i_1$ and we only need to consider $\ell \geq k_1+1$. 
Therefore, it is sufficient to consider $j_2\in [1, k_1]$. 
Note also that $|\{ (-i_1,j_2) \in  [k_2+1, k_1]\times [-k_2, k_1]^*: j_2-i_1=\ell \}|
=c_1(\ell)+c_3(\ell)$. 

\medskip

\begin{lemma} \label{l4}
	\begin{align*}
		S_5=&\sum_{(i_1, j_1)\in \mathcal{X}_2,  (i_2, j_2)\in \mathcal{Y}} 
		\psi(j_1-i_2,-i_1,j_2)+\psi(j_2-i_1,-i_2,j_1)\\
		\leq &   k_2(k_2+1)n+\sum_{\ell=k_1+1}^{2k_1} \sum_{(-i_2,j_1)\in   [-k_2, k_2]^* \times [-k_2, k_2]^*} 
		c_2(\ell) \psi(\ell,  -i_2, j_1).
	\end{align*}
\end{lemma}
\begin{proof} In this case, $-2k_2\leq j_2-i_1\leq 2k_2$. Again, by Lemma \ref{ex1} (a), we only need to consider 
$\ell=j_2-i_1\leq -k_2-1$. Let $c_4(\ell)=|\{ (i_2,j_1) \in [-k_2, k_2]^*\times [-k_2, k_2]^*: j_1-i_2=\ell \}|$. Then
\[ \sum_{\ell=-2k_2}^{-k_2-1}  c_4(\ell)\leq \frac{k_2(k_2+1)}{2}.\]
Note that for $j_1-i_2\leq -k_2-1$, $j_1<0$,  and the number of choices for $i_1$ is $k_2-j_1$. 
However, in the sum, 
\[  \sum_{(-i_1,j_2)\in  [k_2+1, k_1] \times [k_2+1, k_1]} \sum_{\ell=-2k_2}^{-k_2-1} c_4(\ell)\psi(\ell,  -i_1, j_2),\]
a term $\psi(\ell, i, j)$ may appear twice in the sum above. Therefore we have
\begin{align*}
	&\sum_{(-i_1,j_2)\in  [k_2+1, k_1] \times [k_2+1, k_1]} \sum_{\ell=-2k_2}^{-k_2-1} c_4(\ell)\psi(\ell,  -i_1, j_2)\\
	\leq&\sum_{\ell=-2k_2}^{-k_2-1} c_4(\ell)
	\sum_{(i, j)\in \mathcal{Y}\cup \mathcal{Z}} 2 \psi(\ell,  -i_1, j_2)\\
	\leq&2n\sum_{\ell=-2k_2}^{-k_2-1} c_4(\ell)\\
	\leq&k_2(k_2+1)n.
\end{align*}
\end{proof}

\begin{lemma} \label{lc}
\begin{equation} \label{3eq}
\sum_{i=1}^5 S_i \leq (2M+4k_2)(k_1-k_2)n+4n(k_1-k_2)^2+  k_2(k_2+1)n.
\end{equation} 
\end{lemma}
\begin{proof} Note that 
	\begin{align*}
		&S_1+S_2+ \sum_{\ell=k_1+1}^{2k_1} \sum_{(i_2,j_1)\in   [-k_2, k_2]^* \times [k_2+1, k_1]} 2c_1(\ell)\psi(\ell,  -i_2, j_1)\\
		\leq&\sum_{\ell=k_1+1}^{2k_1} 4c_1(\ell) (\sum_{u\in [-k_2, k_1]^*} \psi(\ell, u)+\sum_{(i, j)\in  \mathcal{Y} \cup \mathcal{Z}} \psi(\ell , i, j))\\
		\leq&4n \sum_{\ell=k_1+1}^{2k_1} c_1(\ell)\\
		\leq&4n(k_1-k_2)^2.
	\end{align*}

Next, we collect all summands in the sum $S_3+S_4+S_5$ that involve $c_2(\ell)$. 
\begin{align*}
	&\sum_{\ell=k_1+1}^{2k_1}\sum_{(-i_1,j_2)\in  [k_2+1, k_1]\times [-k_2, k_2]^*} 
	c_2(\ell)\psi(\ell,  -i_1, j_2)+ \sum_{(-i_2,j_1)\in   [-k_2, k_2]^* \times [-k_2, k_2]^*}  c_2(\ell)\psi(\ell,  -i_1, j_2)\\
	\leq& \sum_{\ell=k_1+1}^{2k_1}   2c_2(\ell) \sum_{(i, j)\in \mathcal{Y} \cup \mathcal{Z}} \psi(\ell , i, j)\\
	\leq&2n \sum_{\ell=k_1+1}^{2k_1}   c_2(\ell)\\
	\leq&M(k_1-k_2)  2n.
\end{align*}
Note that a term of the form $\psi(\ell, i, j)$ appear at most twice in the sum above as
$(i, j)$ could be $=(-i_1, j_2)$ or $(j_2, -i_1)$. 
Lastly, we collect summands involve $c_3(\ell)$ in $S_3+S_4+S_5$. 
\begin{align*}
	&\sum_{\ell=k_1+1}^{2k_1}\sum_{(-i_2,j_1)\in  
		[k_2+1, k_1] \times [-k_2, k_1]^*}c_3(\ell) \psi(\ell,  -i_2, j_1)+\sum_{\ell=k_1+1}^{2k_1} \sum_{(i_2,j_1)\in   [-k_2, k_2]^* \times [k_2+1, k_1]} c_3(\ell)) \psi(\ell,  -i_2, j_1)\\
		\leq &\sum_{\ell=k_1+1}^{2k_1} 4c_3(\ell)\sum_{(i, j)\in J} \psi(\ell , i, j)\\
		\leq& 4(k_1-k_2)k_2n.
\end{align*}
As $j_2-i_1\geq k_1+1$, $j_2\geq 1$. For each $j_2$, there are at most $(k_1-k_2+M)$ choices. Therefore $ \sum_{\ell=k_1+1}^{2k_1} c_2(\ell)\leq k_1(k_1-k_2+M)$. 
We then use a similar argument to check that 
$\phi(\ell, -i_1, j_2)$ appears at most three times in the sum. \end{proof}

To compute the sum involving $\psi(j_1-i_2, j_2-i_1)$, it is more convenient to break the summands into three subsums.

\begin{lemma} \label{ld}
	\begin{align*}
	&\frac{1}{2}\sum_{(i_1, j_1)\in \mathcal{X}_1} \sum_{(i_1, j_1)\neq (i_2, j_2)\in \mathcal{X}_1} \psi(j_1-i_2,  j_2-i_1)+\sum_{(i_1, j_1)\in \mathcal{X}_1} \sum_{(i_2, j_2)\in \mathcal{X}_2} \psi(j_1-i_2,  j_2-i_1)\\
	\leq& \frac{1}{2}(k_1-k_2)^2(k_1-k_2+1)n.
	\end{align*}
\end{lemma}
\begin{proof} 
Note that $j_1-i_2\geq k_2+1$ and $j_2-i_1\geq 1$. By Lemma \ref{ex1},
it suffices to consider the case either $j_1-i_2\geq k_1+k_2+1$ or 
$j_2-i_1\geq k_1+k_2+1$.
Note that we take one half of the sum as to avoid double counting. Alternatively, we may also assume $j_1-i_2\geq j_2-i_1$. In that case, $j_2-i_1\leq k_1+k_2$ as 
both $k_2<j_1-i_1\leq k_1+k_2+1$ and $k_2<j_2-i_2 \leq k_1+k_2+1$. Hence, $j_2-i_1$ ranges from $-i_1+k_2+1$ to $\min(-i_1+k_1, k_1+k_2)$. 
Therefore, 
\[ \frac{1}{2}\sum_{(i_1, j_1)\in \mathcal{X}_1} \sum_{ (i_2, j_2)\neq (i_1, j_1)\in \mathcal{X}_1} 
\psi(j_1-i_2,  j_2-i_1)\leq \sum_{\ell=k_1+k_2+1}^{2k_1} c_1(\ell) \sum_{-i_1=k_2+1}^{k_1}  \psi(\ell,  k_2+1-i_1)+\cdots +\psi(\ell, k_1+k_2).\]

Now, consider the case where $(i_2, j_2)\in \mathcal{X}_2$. Then $j_1-i_2\geq 2k_2+2$, and $1\leq j_2-i_1\leq k_2-i_1$. 
\[ \sum_{(i_1, j_1)\in \mathcal{X}_1} \sum_{ (i_2, j_2)\in \mathcal{X}_2} 
	\psi(j_1-i_2,  j_2-i_1)\leq 
	\sum_{\ell=k_1+k_2+1}^{2k_1} c_1(\ell)\sum_{-i_1=k_2+1}^{k_1}\sum_{j_2=-k_2}^{k_2} \psi(\ell,  j_2-i_1).\]
Note that the bound $c_1(\ell)$ used above is not sharp. Therefore, we obtain 
\begin{align*}
	&\frac{1}{2} \sum_{(i_1, j_1)\in \mathcal{X}_1} \sum_{ (i_2, j_2)\neq (i_1, j_1)\in \mathcal{X}_1} 
	\psi(j_1-i_2,  j_2-i_1)+ \sum_{(i_1, j_1)\in \mathcal{X}_1} \sum_{ (i_2, j_2)\in \mathcal{X}_2} 
	\psi(j_1-i_2,  j_2-i_1)\\
	\leq &\sum_{\ell=k_1+k_2+1}^{2k_1} c_1(\ell)\sum_{-i_1=k_2+1}^{k_1}
\psi(\ell, -k_2-i_1)+\cdots \psi(\ell, k_1+k_2) \\
	\leq& \sum_{\ell=k_1+k_2+1}^{2k_1} c_1(\ell)  (k_1-k_2)n
\end{align*}
by Lemma \ref{ex1} (c). 
Finally, for each $j_1\in [k_2+1, k_1]$, the number of choices for $i_2$ is 
$j_1-k_2$ choices for $j_1-i_2$ to fall in the range $[k_1+k_2+1, 2k_1]$. Therefore,
\[ \sum_{\ell=k_1+k_2+1}^{2k_1} c_1(\ell)\leq \frac{(k_1-k_2)(k_1-k_2+1)}{2}.\]
\end{proof}

\begin{lemma} \label{le}
\[ \sum_{(i_1, j_1)\in \mathcal{X}_1} \sum_{(i_2, j_2)\in \mathcal{Y}}\psi(j_1-i_2,  j_2-i_1)\leq  \frac{(k_1-k_2)^2(k_1-k_2+1)n}{2}.\]
\end{lemma}
\begin{proof} Note that $k_1+k_2\geq j_1-i_2\geq 1$ and $1\leq j_2-i_1 \leq 2k_1$. By Lemma \ref{x1} (b), $\psi(j_1-i_2,  j_2-i_1)$ is nonzero only if $j_2-i_1\geq k_1+k_2+1$, which implies $j_2\geq k_2+1$. 
Recall that 
\[ c_3(\ell)=|\{ (j_2, i_1) : j_2\in [k_2, k_1], i_1\in [-k_1, -k_2-1] : 
j_2-i_1=\ell\}|.\] 
Using this, the sum in question becomes:
\begin{align*}
	&\sum_{(i_1, j_1)\in \mathcal{X}_1} \sum_{(i_2, j_2)\in \mathcal{Y}}\psi(j_1-i_2,  j_2-i_1)\\
	\leq&\sum_{\ell=k_1+1}^{2k_1} c_3(\ell) \sum_{j_1=k_2+1}^{k_1} \sum_{i_2\in [-k_2, k_2]^*} \psi(\ell, j_1-i_2)\\
	\leq&\sum_{\ell=k_1+1}^{2k_1} c_3(\ell) (k_1-k_2)n\\
	 \leq&  \frac{(k_1-k_2)^2(k_1-k_2+1)n}{2}.
\end{align*}
\end{proof}

It is more complicated dealing with $|S(i_1, j_1)\cap S(i_1, j_2)|$ when $(i_1, j_1)\in \mathcal{X}_2$ and $(i_2, j_2)\in \mathcal{Y}$. In this case, $j_2-i_1\geq 1$ and
$-2k_2 \leq j_1-i_2 \leq 2k_2$. Note that if $j_1-i_2=0$, we then have the extra term
$(n-1)\psi(j_2-i_1, 0)$. By Lemma \ref{lem3}, $\psi(j_2-i_1,0)\neq 0$ only if $j_2-i_1 \geq k_1+k_2+1$. Recall that $i_1\in [-k_2-M, -k_2-1]$ and hence, we only need to consider
$\ell=k_1+k_2+1$ to $k_1+k_2+M$. By assumption, all such $\ell$ is composite so by Lemma \ref{lem8}, $\psi(\ell, 0)\leq 3$. Let
\[ c_5(\ell)=|\{ (j_2, i_1) : j_2\in [k_2+1, k_1], i_1\in [-k_2-M, -k_2-1] : 
j_2-i_1=\ell\}|.\]
Note that if $j_2=k_1$, there are $M$ choices for $i_1$ and the number of choices reduced by $1$ as $j_2$ reduced by $k_1-M+1$. Hence, 
\[ \sum_{\ell=k_1+k_2+1}^{k_1+k_2+M} c_5(\ell)\leq \frac{M(M+1)}{2}.\]
Therefore, by Lemma \ref{lem8} (b), 
\begin{equation} \label{x1}
\sum_{(i_1, j_1)\in \mathcal{X}_2} \sum_{(i_2, j_2)\in \mathcal{Y}} \psi(j_2-i_1, 0)(n-1)\leq (3/2)M(M+1)n.
\end{equation} 
Next, we deal with 
$\psi(\ell, j_1-i_2)$ with $j_1-i_2\neq 0$. Observe that $j_1-k_2\leq j_1-i_2\leq j_1+k_2$. 
Therefore,
\begin{equation}\label{x3}
\begin{aligned}
& \sum_{(i_1, j_1)\in \mathcal{X}_2} \sum_{(i_2, j_2)\in \mathcal{Y}, i_2\ne j_1}\psi(j_1-i_2,  j_2-i_1)\\
 \leq&\sum_{\ell=k_1+k_2+1}^{k_1+k_2+M} c_5(\ell) \sum_{j_1 \in [-k_2, k_2]^*}  \sum_{j_1\neq i_2\in [-k_2, k_2]^*} \psi(\ell, j_1-i_2)\\
\leq& \sum_{\ell=k_1+k_2+1}^{k_1+k_2+M} c_5(\ell)2k_2n \\
  \leq& k_2M(M+1)n.
\end{aligned}
\end{equation}

Next, we apply our previous results to find a lower bound of $|G'|$. 
Recall that 
\begin{align*}
	|G'|=&|G-\bigcup_{(i,j)\in \mathcal{Z}} S(i,j)|+|\bigcup_{(i,j)\in \mathcal{X}} S(i,j)|
	-| (\bigcup_{(i,j)\in \mathcal{X}} S(i,j))\cap (G-\bigcup_{(i,j)\in \mathcal{Y}} S(i,j))|\\
	&- | (\bigcup_{(i,j)\in \mathcal{X}} \bigcup_{u\in [-k_2, k_1]^*} |S(i,j))\cap T^{(u)}|.
\end{align*}
Applying Lemmas \ref{lc}, \ref{ld}, \ref{le}  and inequalities (\ref{x1}), (\ref{x3}) and that the smallest prime factor of $k_1+k_2+1$ is $p$, we then obtain 
\begin{equation}
\begin{aligned}
	& \frac{k_1+(4M-1)k_2}{2}(n^2-n) <  [(2M+4k_2)(k_1-k_2)n+4(k_1-k_2)^2n+  k_2(k_2+1)n \\
& 
+\frac{(k_1-k_2)3\sqrt{p}n}{2}
 +(k_1-k_2)^2(k_1-k_2+1)n +(3/2)M(M+1)n+ k_2M(M+1)n.
\end{aligned}
\end{equation}
Set
\begin{align*}
	A=&k_1+(4M-1)k_2,\\
	B=&(4M+8k_1+3\sqrt{p})(k_1-k_2)+2k_2(k_2+1)+2(k_1-k_2)^2(k_1-k_2+1)+(3+2k_2)M(M+1).
\end{align*} 
We conclude that if $n\ge\frac{B}{A}+1$, then 
no lattice tiling of $\mathbb{Z}^n$ by $\mathcal{B}(n,2,k_1,k_2)$ exists.
This completes the proof of Theorem~\ref{mainthm1}.

\medskip

To get a good understanding how large $B/A$ can be, we set  
$M=1$ and that $3\sqrt{p}\leq 3\sqrt{k_1}$. We then obtain 
\[ \frac{B}{A}+1\le 2(k_1-k_2)^2+12(k_1-k_2)+ 2k_2+8.\]
This proves Corollary~\ref{coro1}.

\section{Conclusion}\label{conclu}
This work presents three advances in the theory of lattice tilings with limited magnitude error balls:
\begin{enumerate}
	\item \textbf{Definitive classification for $k_1=3$ and $k_2=0$:} For $\mathcal{B}(n,2,3,0)$, we fully characterize lattice tiling existence:
	\begin{itemize}
		\item Non-existence for all $n\ge4$,
		\item Explicit construction for $n=3$ via generator $T=\{1,10,26\}$ in $\mathbb{Z}_{37}$.
	\end{itemize}
	\item \textbf{Complete resolution for $k_1=k_2+1$:} We establish that no lattice tiling of $\mathbb{Z}^n$ by $\mathcal{B}(n,2,k,k-1)$ for all $n\ge3$ and $k\ge2$.
	\item \textbf{General non-existence theorem:} For any integer $k_1>k_2\ge0$ with $k_1+k_2+1$ is composite, we demonstrate the non-existence of lattice tilings of $\mathbb{Z}^n$ by $\mathcal{B}(n,2,k_1,k_2)$ in all sufficiently high dimensions $n$.
\end{enumerate}
For the case where $k_1+k_2+1=p$ is a prime number, the problem becomes more intricate, and our current method can be applied only if $G$ is not $p$-elementary.
In other words, our method works if $|G|$ is not a $p$-power. 
Furthermore, our approach extends naturally to higher values of $t$ (i.e., $t\ge3$), which we plan to explore in future work.

\section*{Acknowledgements}
Ran Tao is partially supported by the National Key Research and Development Program of China (Grant No. 2022YFA1004900), the National Natural Science Foundation of China (Grant No. 12401437), and the Shandong Provincial Natural Science Foundation (Grant No. ZR2022QA069).
Tao Zhang is partially supported by the Fundamental Research Funds for the Central Universities (Grant No. QTZX24082), and Natural Science Basic Research Program of Shaanxi  (Program No. 2025JC-YBMS-048).

\end{document}